\documentclass{article}
\usepackage{amsmath,amsthm,amssymb,graphics,wrapfig}
\newcommand{\R}{\mathbb{R}\mkern1mu}

\newcommand{\N}{\mathbb{N}\mkern1mu}
\newtheorem{theorem}{Theorem}
\newtheorem{lemma}{Lemma}
\newtheorem{cor}{Corollary}

\theoremstyle{remark}

\title{Geometrically $L^p$-optimal lines of vertices of an equilateral triangle}
\author{Annett P\"uttmann}
\date{\today}
\begin{document}
\maketitle 
\begin{abstract}
We consider the distances between a line and a set of points in the plane defined by
the $L^p$-norms of the vector consisting of the euclidian distance between the single
points and the line.
We determine lines with minimal geometric $L^p$-distance to the vertices of an equilateral triangle
for all $1\leq p\leq \infty$.
The investigation of the $L^p$-distances for $p\ne 1,2,\infty$ establishes the passage between the
well-known sets of optimal lines for $p=1,2,\infty$.

The set of optimal lines consists of three lines each parallel to one of the triangle sides 
for $1\leq p<4/3$ and $2<p\leq \infty$ and 
of the three perpendicular bisectors of the sides for $4/3<p<2$. 
For $p=2$ and $p=4/3$ there exist one-dimensional families of  optimal lines.
\end{abstract}
\section{Introduction}
In order to investigate the problem of finding lines in the plane which are as close as possible to a given
finite set of points $P:=\{Êp_1, \ldots, p_m\} \subset \R^2$ it is necessary to define the distance between
a line  $g\subset\R^2$ and the set $P$.
A suitable notion of distance depends on the specific problem which motivated the interest in lines
close to given points. 
It is often useful to define such a distance in two steps:
First the distance between a single point $p_j$ and a line $g$ is defined.
We call these distances  $d_j$. 
Then the distances $d_j$ are combined to a notion of distance between the line $g$ and the set $P$
which we denote by $d(g,P)$.

It is appropriate to work with the algebraic (vertical) distance between a point and a line
for the interpolation by functions, e.g., linear regression.
The geometric (euclidian) distance between a point and a line is often used in optimization problems
and even in statistics  \cite{C} , \cite{M}.
In this article $d_j$ is always the geometric distance. 

Any norm on $\R^m$ leads to a definition of $d(g,P)$.
In particular, 
$$d(g,P) := \| (d_1,\ldots,d_m)\|_p = \| d\|_p = \left( \sum_{j=1}^m d_j^p\right)^{1/p} \text{Êfor } 1\leq p $$
and
$$d(g,P) = \| (d_1,\ldots,d_m)\|_\infty = \| d\|_\infty = \max\{ d_j : 1\leq j\leq m\} \text{ for } p=\infty$$
are invariant under permutations of the set $P$. 
Furthermore, the summands could be weighted to obtain non-symmetric distances \cite{S}.
In this article we consider the symmetric distances given by $\|d\|_p$ for $1\leq p \leq \infty$.
	
The $L^2$-norm, that is frequently used, corresponds to the method of least squares \cite{C}.
The $L^2$-norm is among all $L^p$-norms the only norm which is implied by an inner product.
This simplifies many calculations. 
In statistic the $L^2$-norm occurs very often, e.g., in the case of linear regression,
since the minimum with respect to the $L^2$-norm coincides with the maximum-likelihood estimate 
for normally distributed random variables.
The $L^1$-norm appears frequently in optimization problems \cite{S}
and is also used in statistics investigating least absolute deviations  \cite{BS}.
The $L^\infty$-norm is appropriate to measure the quality 
of an interpolation by functions \cite{S},\cite{SW}.
\subsection{Motivation for $L^p$-norms with arbitrary $1\leq p\leq\infty$}
Concentrating on the pure optimization problem for a generic point set we notice two things:
On the one hand the optimal lines for $p=1$ $p=2$ and $p=\infty$ are different and even far apart
from each other in the parameter space of lines 
(see subsection \ref{sec:BekannteResultateP=1,2,unendlich} 
and Figures \ref{Abb:Optp1} and \ref{Abb:Optp-max}).
On the other hand the function to be minimized, $\|d\|_p$, is continuous in $p$ and
in the parameters of the line.
Hence, the $L^p$-optimal lines move through the set of given points as $p$ changes. 
Our interest is mainly in the explicit determination of the $L^p$-optimal lines and their properties 
and not in the minimal $L^p$-distance.
In particular, we want to observe the limits to the extreme norms, $p\to 1$ and $p\to\infty$, and
to the euclidian norm for optimal lines.
It is sufficient to minimize the functions $f_p:= d(g,P)^p$ for $1\leq p <\infty$ in order to find the lines
with minimal distance $d(g,P)=\|d\|_p$.
\subsection{Formulation of the problem}
We investigate the simplest geometric non-trivial situation, i.e., $m=3$, 
the points $p_1,p_2,p_3$ are the vertices of an equilateral triangle $D$ 
and $d_j$ is the euclidian distance between a point $p_j$ and  a line $g\subset \R^2$.
We determine the global minima of the functions
$$f_p:=\sum_{j=1}^m d_j^p \text{ for } 1\leq p<\infty \text{ and } f_\infty :=\max\{ d_j :j=1,\ldots,m\}$$ 
and the corresponding lines for all $1\leq p\leq\infty$. 
These lines are called {\it $L^p$-optimal} or simply {\it optimal} lines.
In particular, we obtain results on the dependence of the set of $L^p$-optimal lines 
on the parameter $p$.
\subsection{Sketch of the solution}
The set of optimal lines is invariant under the symmetry group of the triangle $D$ for every $p$.
Since lines are not  invariant under a rotation through an angle of $2\pi/3$,
there exist at least three optimal lines. 

The relevant known results for the cases $p=1$, $p=2$ and $p=\infty$ are stated
in subsection \ref{sec:BekannteResultateP=1,2,unendlich}.
In section \ref{sec:EigenschaftenOptimaler Geraden} we show general properties of optimal lines 
with respect to three points for $1<p<\infty$. 
Combining these facts with the symmetries of the triangle $D$ we reduce the domain of $f_p$ 
in  section \ref{sec:Reduktion} to an even smaller compact set $M$. 
We prove absence of critical points of $f_p$ in the interior of $M$ for $p\ne 2, 4/3$ 
in subsection \ref{subsec:keine-kritischen-Punkte}.
This proof is the essential ingredient of the solution.
For $p=2$ and $p=4/3$ there exist one-dimensional submanifolds of critical points 
which intersect the boundary of $M$. 
Investigating the functions $f_p$ on the boundary of $M$ in subsection \ref{subsec:Rand}
we are able to determine the minima of the functions $f_p$ on $M$ exactly.
\subsection{Known results for $p\in\{ 1,2,\infty\}$}
\label{sec:BekannteResultateP=1,2,unendlich}
The lines with minimal $L^p$-distance to an arbitrary finite set $\subset\R^2$ are known for 
$p=1$, $p=2$ and $p=\infty$ \cite{C}, \cite{S}.
See also  \cite{AP} for a self-contained introduction to the problem of lines with minimal algebraic or
geometric $L^p$-distance to a finite set in the plane for $1\leq p\leq \infty$.
The first descriptions of $L^2$-optimal lines can be found in  \cite{A}, \cite{K}, \cite{P} and \cite{M}.
Laplace \cite{BS} has already solved the problem of algebraically $L^1$-optimal lines. 
These ideas can be adapted to determine geometrically $L^1$- and $L^\infty$-optimal lines exactly. 

We denote the length of the sides of the triangle $D$ by $s$ in this subsection.
\begin{figure}
\noindent
\begin{minipage}[b]{.49\linewidth}
  \mbox{\scalebox{.28}{\includegraphics{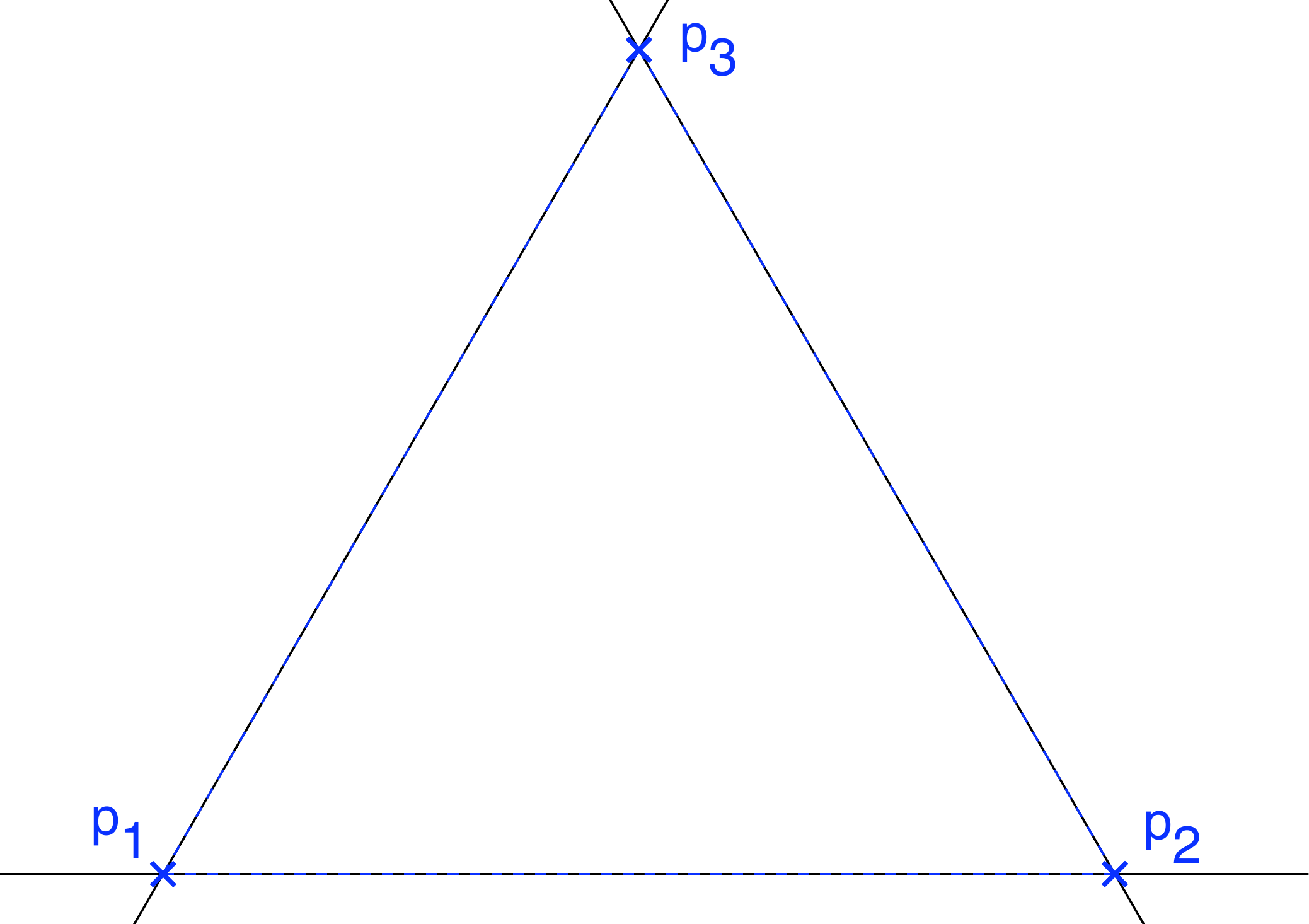}}}
\caption{$p=1$}
\label{Abb:Optp1}
\end{minipage}
\begin{minipage}[b]{.49\linewidth}
  \mbox{\scalebox{.28}{\includegraphics{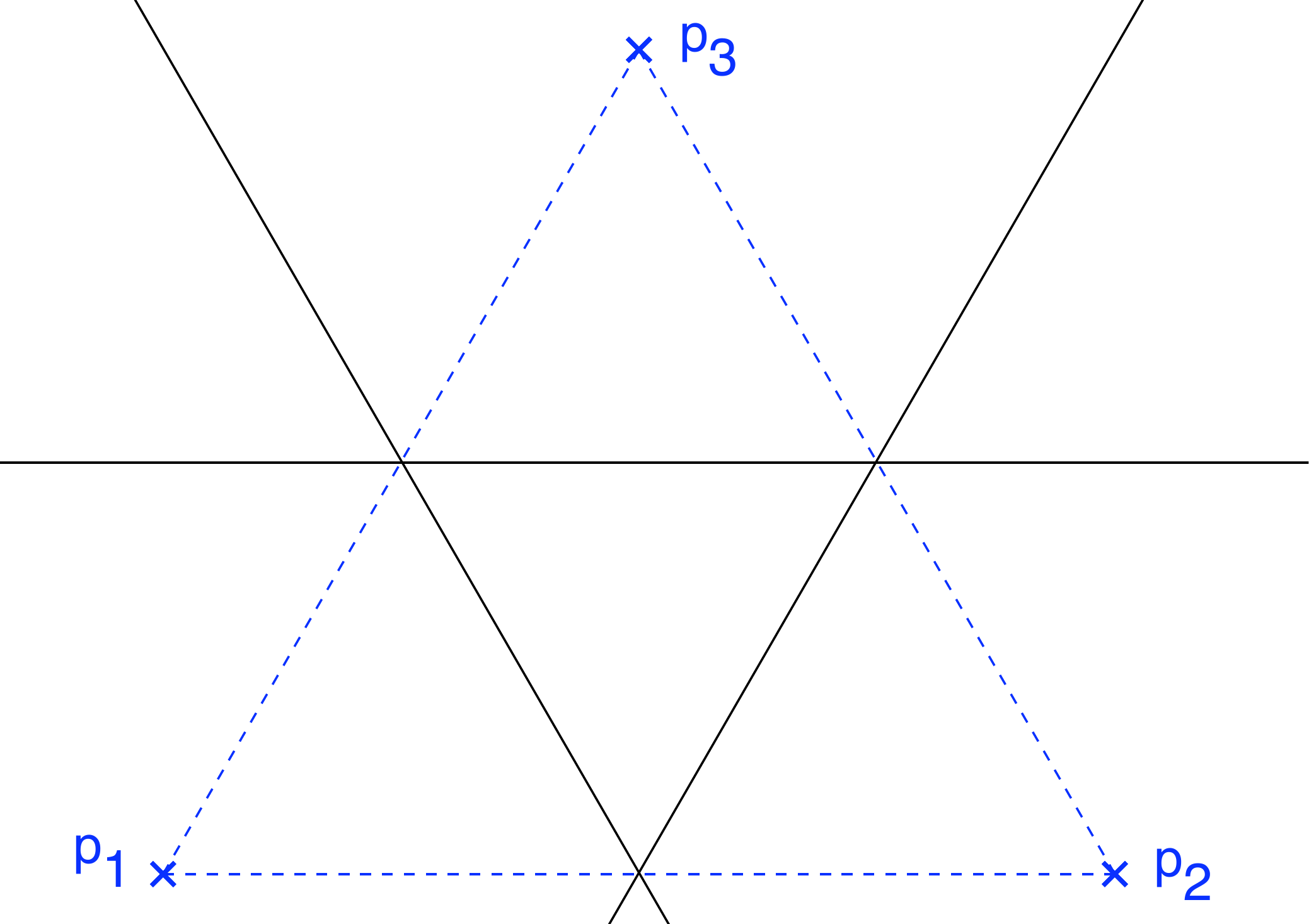}}}
\caption{$p=\infty$}
\label{Abb:Optp-max}
\end{minipage}
\end{figure}
\subsubsection{Absolute geometric distance ($p=1$)}
The minimum of the function $f_1=\sum_{j=1}^m d_j$ is attained at lines 
containing at least two of the points $p_j$. 
A line containing exactly one of the points $p_j$ is never optimal.
A line $g$ containing none of the points $p_j$ is optimal if and only if 
there exist optimal lines $g_1$ and $g_2$ each containing two of the points $p_j$ and parallel to $g$ 
such that $g$ but none of the points $p_j$ lie between $g_1$ and $g_2$.

Since the three sides of the triangle $D$ are not parallel, 
the optimal lines are exactly the three lines containing two of the points $p_1,p_2,p_3$ 
(see Figure \ref{Abb:Optp1}).
Hence, the global minimum of the function $f_1=d_1+d_2+d_3$ is $\sqrt{3}s/2$, i.e.,
the length of the height of the triangle $D$.
\subsubsection{Geometric least squares ($p=2$)}
The minimum of the function $f_2=\sum_{j=1}^m d_j^2$ is attained at a line $g$ if and only if
$g$ contains the center of mass $\bar{p}$ of the set $\{p_1,\ldots,p_m\}$ and 
a normal vector of $g$ is an eigenvector of the smallest eigenvalue of the symmetric matrix
$S=\sum_{j=1}^m(p_j-\bar{p}) (p_j-\bar{p})^T$.

In our situation the set $\{p_1,p_2,p_3\}$ and $S$ are invariant under rotations around $\bar{p}$ 
through an angle of $2\pi/3$. Since the eigenspaces of the symmetric matrix $S$ are perpendicular,
$S$ has a two-dimensional eigenspace. 
This means that the optimal lines are exactly the lines containing $\bar{p}$ (see Figure \ref{Abb:OptB1}). The global minimum of the function $f_2=d_1^2+d_2^2+d_3^2$ is $s^2/2$.
\subsubsection{Maximal geometric distance ($p=\infty$)}
The minimum of the function $f_\infty = \max\{ d_j : j=1,\ldots m\}$ is attained at a line $g$ 
if and only if $g$ has the following properties:
There exists a line $l$ parallel to $g$ containing two of the points $p_j$.
There exists a point $p_k\not\in l$ such that the geometric distance between $g$ and $l$ is equal
to $d_k$ and $d_j\leq d_k$ for all $j$.

Hence, a line $g$ has minimal $L^\infty$-distance to the vertices $p_1,p_2,p_3$ 
of an equilateral triangle $D$ if and only if $d_1=d_2=d_3$ (see Figure \ref{Abb:Optp-max}).
The minimum of the function $f_\infty=\max\{d_1,d_2,d_3\}$ is $\sqrt{3}s/4$.
\section{Properties of optimal lines for $p\ne 1,\infty$}
\label{sec:EigenschaftenOptimaler Geraden}
A line $g\in\R^2$ is completely characterized by a normal vector $n\in S^1$ and a point $q_0\in g$,
i.e., $g = \{Êq\in\R^2:\langle n,q\rangle = \langle n,q_0\rangle\}$. Set $c:= \langle n,q_0\rangle\in\R$.
The geometric distance between $g$ and $p_j$ is given by $d_j=|c-\langle n,p_j\rangle|$.
Hence, we investigate the function
$$f(c,n) := f_p(c,n) = \sum_{j=1}^3 |c-\langle n,p_j\rangle|^p 
= \sum_{j\in J_+}(\langle n,p_j\rangle-c)^p+\sum_{j\in J_-}(c-\langle n,p_j\rangle)^p$$
where the decomposition of the index set $\{1,2,3\} = J_+\cup J_0\cup J_-$ is defined by
$J_+:=\{Êj:\langle n,p_j\rangle>c\}$, $J_0=\{ j:\langle n,p_j\rangle = c\}$ and
$J_-:=\{ j:\langle n,p_j\rangle<c\}$.
\begin{lemma}
Let $1<p<\infty$. If  $c=\langle n,q\rangle$ is an $L^p$-optimal line, then 
\begin{equation}
\label{eq:partielle Ableitung nach c}
\sum_{j\in J_+}d_j^{p-1} = \sum_{j\in J_-}d_j^{p-1}.
\end{equation}
\end{lemma}
\begin{proof}
The function $f$ is differentiable.
If $J_0 = \{Êj:\langle n,p_j\rangle = c\}$, then 
\begin{align*}
f(c+\varepsilon, n) -f(c,n) & = |\varepsilon|^p |J_0| 
	+ \sum_{j\in J_+}(\langle n,p_j\rangle-c-\varepsilon)^p-(\langle n,p_j\rangle-c)^p \\
& \quad + \sum_{j\in J_-}(c+\varepsilon-\langle n,p_j\rangle)^p- (c-\langle n,p_j\rangle)^p\\
\lim_{\varepsilon\to 0} \frac{f(c+\varepsilon+c,n)-f(c,n)}{\varepsilon} & = 
	\left.\frac{\partial f}{\partial c}\right\vert_{(c,n)} 
	= p \left(\sum_{j\in J_-} d_j^{p-1}-\sum_{j\in J_+} d_j^{p-1}\right)
\end{align*}
for all $\varepsilon\in\R$ with $|\varepsilon|<\min\{ d_j:j\not\in J_0\}$.
\end{proof}
\begin{cor}
\label{folg:Zerlegung der Punktmenge durch optimale Gerade}
It holds $J_+\ne \emptyset$, $J_-\ne \emptyset$ and $|J_0|\leq 1$ for any optimal line.
\end{cor}
\begin{proof}
The points $p_1,p_2,p_3$ are not collinear. Hence, $J_+\cup J_- \ne \emptyset$.
The assertion follows from equation (\ref{eq:partielle Ableitung nach c}), 
because $d_j=0$ if and only if $j\in J_0$, 
$\sum_{j\in J_+}d_j^{p-1} = 0$ if and only if $J_+=\emptyset$, and
$\sum_{j\in J_-}d_j^{p-1} = 0$ if and only if $J_-=\emptyset$.
\end{proof}
\begin{cor}
\label{folg:Abstand Punkte zur optimalen Geraden}
If $J_0=\emptyset$ for an optimal line, then there exists a permutation $\sigma$ such that
$J_+=\{ \sigma(1),\sigma(2)\}$, $J_-=\{ \sigma(3)\}$ 
or $J_-=\{ \sigma(1),\sigma(2)\}$, $J_+=\{ \sigma(3)\}$ and
$$ d_{\sigma(3)}> d_{\sigma(2)} \geq d_{\sigma(1)} >0 .$$
\end{cor}
\begin{proof}
One side of equation (\ref{eq:partielle Ableitung nach c}) consists of exactly one summand
$d_{\sigma(3)}^{p-1}$. 
The other side of the equation is of the form $d_{\sigma(1)}^{p-1} + d_{\sigma(2)}^{p-1}$
with $0<d_{\sigma(1)}^{p-1} \leq d_{\sigma(2)}^{p-1}$, since $J_0=\emptyset$.
Consequently, $d_{\sigma(1)}^{p-1} < d_{\sigma(3)}^{p-1}$ and 
$d_{\sigma(2)}^{p-1}  < d_{\sigma(3)}^{p-1}$. 
Now, $p>1$ implies $d_{\sigma(1)} \leq d_{\sigma(2)}$ and $d_{\sigma(2)} < d_{\sigma(3)}$.
\end{proof}
\begin{cor}
\label{folg:1Punkt-optimal}
If an optimal line contains one of the points $p_j$, 
then this line is a perpendicular bisector of the triangle $D$.
\end{cor}
\begin{proof}
The condition $|J_0|\geq 1$ and Corollary \ref{folg:Zerlegung der Punktmenge durch optimale Gerade} 
imply $|J_0|=|J_+|=|J_-|=1$. Hence, there exists a permutation such that 
$J_0=\{\sigma(1)\}$, $J_+=\{\sigma(2)\}$ and $J_-=\{\sigma(3)\}$. 
Equation \ref{eq:partielle Ableitung nach c} implies $d_{\sigma(2)}^{p-1}=d_{\sigma(3)}^{p-1}$. 
Consequently, $d_{\sigma(2)}=d_{\sigma(3)}$ since $p>1$.
\end{proof}
\section{Reduction}
\label{sec:Reduktion}
\begin{figure}
\begin{center}
  \mbox{\scalebox{.35}{\includegraphics{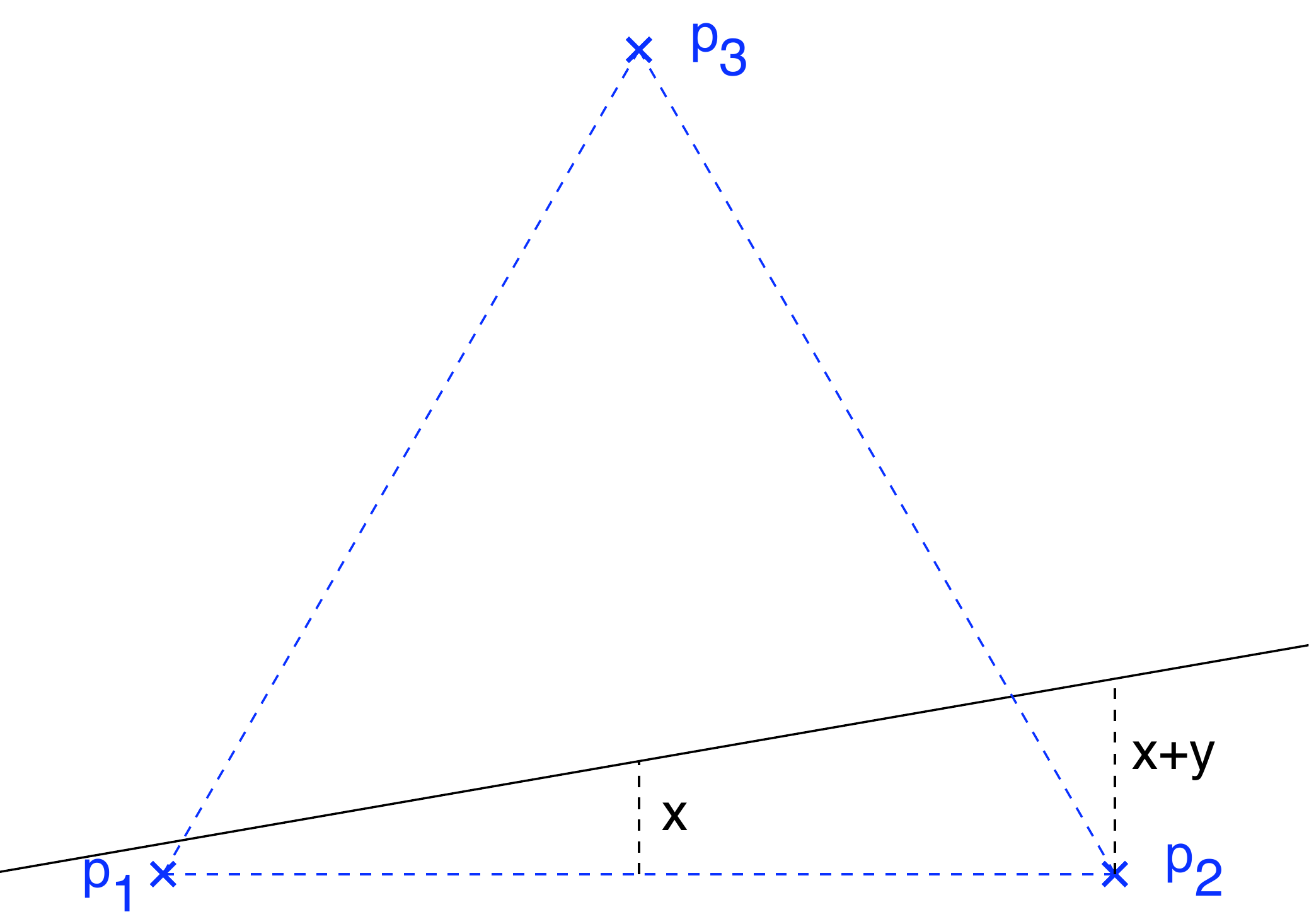}}}
\caption{}
\label{Abb:OptiRedukt}
\end{center}
\end{figure}
The set of optimal lines is equivariant with respect to isometries and dilations.
Hence, we assume that $p_1=(-1/2,0)$, $p_2=(1/2,0)$, $p_3=(0,\sqrt{3}/2)$.
In particular, $s=1$.
Due to the rotation symmetry of $D$ it is sufficient to find optimal lines which intersect the sides
$p_1p_3$ and $p_2p_3$. 
The reflection symmetry of $D$ allows us to assume $d_1\leq d_2$ additionally
(see Figure \ref{Abb:OptiRedukt}).
\begin{lemma}
Let $g$ be an optimal line.
If $J_-=\{ p_1,p_2\}$, $J_+=\{ p_3\}$ and $(0,x)\in g$, then $0<x<\sqrt{3}/4$.
\end{lemma}
\begin{proof}
Let the optimal line $g$ be given by the equation $c=\langle n,q\rangle$ 
with $n=(n_1,n_2)$ and $c\in\R$. The condition  $(0,x)\in g$ implies $c=n_2x$. 
Furthermore, $-n_1/2<c$ and  $n_1/2<c$, since $J_-=\{p_1,p_2\}$. 
Thus, $0\leq |n_1/2|<c = n_2x$. 
The condition $J_+=\{ p_3\}$ implies $n_2\sqrt{3}/2>c>0$. Hence, $n_2>0$ and $x>0$.

Corollary \ref{folg:Abstand Punkte zur optimalen Geraden} yields the inequality
$\max\{ d_1,d_2\} < d_3$ for any optimal line $g$.
It follows from  $\max\{ d_1,d_2\} = c+|n_1|/2 = n_2x+|n_1|/2$ and  $d_3=n_2(\sqrt{3}/2-x)$ that
$|n_1|<n_2(\sqrt{3}-4x)$. Now $n_2>0$ implies $x<\sqrt{3}/4$.
\end{proof}
It is sufficient to consider lines containing the points $(0,x)$ and $(1/2,x+y)$ 
such that $0<x<\sqrt{3}/4$ and $0\leq y<x$ (see Figure \ref{Abb:OptiRedukt}) to find optimal lines. 
Such a line is spanned by the vector $(1,2y)$ and it is given by the equation 
$c=\langle n,q\rangle$ with normal vector $n=(-2y,1)(1+4y^2)^{-1/2}$ and $c=x(1+4y^2)^{-1/2}$.
The geometric distances between the line and the points $p_1,p_2,p_3$ are
$d_1= (x-y)(1+4y^2)^{-1/2}$, $d_2=(x+y)(1+4y^2)^{-1/2}$ and 
$d_3=(\sqrt{3}/2-x)(1+4y^2)^{-1/2}$.
We want to determine the global minimum of the function
\begin{multline}
\label{eq:f(x,y) und M}
f(x,y)= (1+4y^2)^{-p/2}\left( (x-y)^p+(x+y)^p+\left(\frac{\sqrt{3}}{2}-x\right)^p\right)\\
\text{ on the set } M:=\{0\leq x\leq \sqrt{3}/4, \, 0\leq y\leq x\}.
\end{multline}
We already know that the minimum of $f$ is not attained at points with $x=0$ or $x=\sqrt{3}/4$
for $1<p<\infty$. 
Corollary \ref{folg:1Punkt-optimal} implies that global minimum of $f$ on  the boundary component
$\{ 0\leq x=y \leq \sqrt{3}/4\}$ is attained only at $x=y=\sqrt{3}/6$.
\section{Global minimum on a compact set}
In this section we determine the global minimum of the function $f$ on the set $M$ for all $1<p<\infty$.
We partially solve the system of equations defining critical points of $f$ 
in subsection \ref{subsec:Gleichungssystem-kritische, innere Punkte}.
In subsection \ref{subsec:keine-kritischen-Punkte} we show that $f$ has critical points in the interior
$M^\circ$ of $M$ if and only if $p=4/3$ or $p=2$.
Comparing the local minima of $f$ on the different boundary components of $M$ in subsection
\ref{subsec:Rand} we obtain explicit formulas of the global minimum of $f$ on $M$.
\subsection{Equations characterizing critical points}
\label{subsec:Gleichungssystem-kritische, innere Punkte}
We investigate $f(x,y)$ on $M^\circ=\{Ê(x,y)\in\R^2: 0<x<\sqrt{3}/4, 0<y<x\}$.
\begin{lemma}
\label{lem:krtischerPunkt-Gleichungen}
It holds $df = (0,0)$ if and only if
\begin{equation}
\label{eq:df=0(A)}
\left(\frac{\sqrt{3}}{2}-x\right)^{p-1} = (x-y)^{p-1}+(x+y)^{p-1} 
\end{equation}
and
\begin{equation}
\label{eq:df=0(B)}
0 = (x+y)^{p-1}\left(2\sqrt{3}y-1\right) +(x-y)^{p-1}\left(2\sqrt{3}y+1\right).
\end{equation}
\end{lemma}
\begin{proof}
The function $f$ is differentiable. The partial derivatives are
$$f_x(x,y) = p(1+4y^2)^{-p/2}\left( (x-y)^{p-1}+(x+y)^{p-1}-\left(\frac{\sqrt{3}}{2}-x\right)^{p-1}\right) $$
and
\begin{multline*}
f_y(x,y) = -4py(1+4y^2)^{-p/2-1}\left( (x-y)^p+(x+y)^p+\left(\frac{\sqrt{3}}{2}-x\right)^p\right) \\
+p(1+4y^2)^{-p/2}\left( -(x-y)^{p-1}+(x+y)^{p-1}\right)
\end{multline*}
It holds $f_x=f_y=0$ if and only if
\begin{equation}
\label{eq:fx=0}
\left(\frac{\sqrt{3}}{2}-x\right)^{p-1} = (x-y)^{p-1}+(x+y)^{p-1} 
\end{equation}
and
\begin{equation}
\label{eq:fy=0}
\frac{4y}{1+4y^2}\left( (x-y)^p+(x+y)^p+\left(\frac{\sqrt{3}}{2}-x\right)^p\right) 
	= (x+y)^{p-1}-(x-y)^{p-1}.
\end{equation}
We use equation (\ref{eq:fx=0}) to replace the term $(\sqrt{3}/2-x)^{p-1}$ in equation (\ref{eq:fy=0}).
Thus, 
\begin{align*}
& (1+4y^2)\left((x+y)^{p-1}-(x-y)^{p-1} \right)\\ 
& = 4y\left( (x-y)^p+(x+y)^p+\left(\frac{\sqrt{3}}{2}-x\right)\left( (x-y)^{p-1}+(x+y)^{p-1} \right)\right) \\
& =  (x-y)^{p-1}\left(2\sqrt{3}y-4y^2\right)+(x+y)^{p-1}\left(2\sqrt{3}y+4y^2\right) \\
0 & = (x+y)^{p-1}\left(2\sqrt{3}y-1\right) +(x-y)^{p-1}\left(2\sqrt{3}y+1\right)
\end{align*}
\end{proof}
\begin{cor}
\label{folg:kritischer Punkt, Intervall fuer y}
If $(x,y)\in M^\circ$ is a critical point of $f$, then $y < \sqrt{3}/6$.
\end{cor}
\begin{proof}
If $(x,y)\in M^\circ$, then $x+y>0$, $x-y>0$ and $1+2\sqrt{3}y>0$. 
Equation (\ref{eq:df=0(B)}) is only satisfied if $2\sqrt{3}y-1<0$, i.e., $y<\sqrt{3}/6$.
\end{proof}
\begin{cor}
\label{folg:kritischer Punkt, zwei Gleichungen}
An inner point $(x,y)\in M^\circ$ is a critical point of $f$ if and only if
\begin{equation}
\label{eq:kritischer Punkt, x=x(y)}
x = y \frac{\left(1+2\sqrt{3}y\right)^b+\left(1-2\sqrt{3}y\right)^b}{\left(1+2\sqrt{3}y\right)^b 
	-\left(1-2\sqrt{3}y\right)^b} \;  \text{Êwith } b:=\frac{1}{p-1},
\end{equation}
and
\begin{equation}
\label{eq:kritischer Punkt, Bedingung in t}
0 = 2^{b+1}t+3\left(\left(1-t\right)^b -\left(1+t\right)^b\right)+t\left(\left(1+t\right)^b +\left(1-t\right)^b\right)
\; \text{Êwith } t:=2\sqrt{3}y.
\end{equation}
\end{cor}
\begin{proof}
If $(x,y)\in M^\circ$ is a critical point of $f$, then equation (\ref{eq:df=0(B)}) holds.
The variable $x$ can be eliminated from equation (\ref{eq:df=0(B)}), since $y<\sqrt{3}/6$: 
\begin{align*}
 (x+y)^{p-1}\left(1-2\sqrt{3}y\right) & = (x-y)^{p-1}\left(2\sqrt{3}y+1\right) \\
(x+y)\left(1-2\sqrt{3}y\right)^{1/(p-1)} & = (x-y)\left(1+2\sqrt{3}y\right)^{1/(p-1)}\\
x & = y \frac{\left(1+2\sqrt{3}y\right)^b+\left(1-2\sqrt{3}y\right)^b}{\left(1+2\sqrt{3}y\right)^b 
	-\left(1-2\sqrt{3}y\right)^b} \text{ with } b=\frac{1}{p-1}
\end{align*}
Now
$$x-y = 2y \frac{\left(1-2\sqrt{3}y\right)^b}{\left(1+2\sqrt{3}y\right)^b -\left(1-2\sqrt{3}y\right)^b},\,
x+y = 2y \frac{\left(1+2\sqrt{3}y\right)^b}{\left(1+2\sqrt{3}y\right)^b -\left(1-2\sqrt{3}y\right)^b} $$
and equation (\ref{eq:df=0(A)}) becomes 
\begin{align*}
\left(\frac{\sqrt{3}}{2}-x\right)^{p-1}
& = (2y)^{p-1}\frac{2}{\left(\left(1+2\sqrt{3}y\right)^b -\left(1-2\sqrt{3}y\right)^b\right)^{p-1}} \\
\frac{\sqrt{3}}{2}-x & = \frac{2^{b+1}y}{\left(1+2\sqrt{3}y\right)^b -\left(1-2\sqrt{3}y\right)^b} \\
\frac{\sqrt{3}}{2} & = y \frac{2^{b+1} +\left(1+2\sqrt{3}y\right)^b + \left(1-2\sqrt{3}y\right)^b }{\left(1+2\sqrt{3}y\right)^b -\left(1-2\sqrt{3}y\right)^b}.
\end{align*}
Multiplication with the denominator and the substitution $t:=2\sqrt{3}y$ lead to
{\allowdisplaybreaks
\begin{align*}
0 & = 2^{b+1}y +\frac{\sqrt{3}}{2}\left(1-2\sqrt{3}y\right)^b -\frac{\sqrt{3}}{2}\left(1+2\sqrt{3}y\right)^b \\
	& \quad+y\left(1+2\sqrt{3}y\right)^b +y\left(1-2\sqrt{3}y\right)^b \\
0 & = 2^{b+1}2\sqrt{3}y +3\left(1-2\sqrt{3}y\right)^b -3\left(1+2\sqrt{3}y\right)^b \\
	& \quad+2\sqrt{3}y\left(1+2\sqrt{3}y\right)^b +2\sqrt{3}y\left(1-2\sqrt{3}y\right)^b \\
0 & = 2^{b+1}t+3\left(1-t\right)^b -3\left(1+t\right)^b+t\left(1+t\right)^b +t\left(1-t\right)^b.
\end{align*}
}
\end{proof}
We have to solve equation (\ref{eq:kritischer Punkt, Bedingung in t}) to find the critical points of $f$.
This means that we have to find the zeros of the function $g$ defined by
$$ g(t):=2^{b+1}t+3\left(\left(1-t\right)^b -\left(1+t\right)^b\right)
	+t\left(\left(1+t\right)^b +\left(1-t\right)^b\right)$$
in the intervall Intervall $(0,1)$ for all $b\in\R$ with $b>0$.
Note that $g(0)=0$ und $g(1) = 0$ for all $b>0$.
\begin{lemma}
If $b=1$ or $b=3$, then $g\equiv 0$.
\end{lemma}
\begin{proof}
For $b=1$ we check that
$$g(t) = 4t+3\left(\left(1-t\right) -\left(1+t\right)\right)+t\left(\left(1+t\right)+\left(1-t\right)\right)
	= 4t-6t+2t\equiv 0.$$
For $b=3$ we check that
\begin{align*}
g(t) & =16t+3\left(\left(1-t\right)^3 -\left(1+t\right)^3\right)+t\left(\left(1+t\right)^3+\left(1-t\right)^3\right) \\
& = 16t + 3(-6t-2t^3)+t(2+6t^2) \equiv 0.
\end{align*}
\end{proof}
\subsection{Absence of interior critical points for $b\ne 1,3$}
\label{subsec:keine-kritischen-Punkte}
We show that the function $g$ has no zeros in the intervall $(0,1)$ for all $b\ne 1,3$. 
We consider the function defined by $h(t):=g(t)/t$ for $0<t\leq 1$. It holds $h(1)=0$ and
$\lim_{t\to 0} h(t) = 2^{b+1}-6b+2=2(2^b-3b+1)$.
\begin{lemma}
The function $s:\R\to\R$ defined by $s(b):=2^b-3b+1$ has the following properties:
If $b<1$ or $b>3$ then $s(b)>0$, if $1<b<3$ then $s(b)<0$, $s(1)=s(3)=0$.
\end{lemma}
\begin{proof}
It is easy to check that $s(1)=2-3+1=0$ and $s(3)=2^3-9+1=0$.
Moreover, the function $s$ is convex, since $s''(b) = 2^b (\ln 2)^2>0$ for all $b\in \R$.
\end{proof}
Expanding $(1-t)^b$ and $(1+t)^b$ into power series for $-1<t<1$, i.e.,
$$(1-t)^b=\sum_{n=0}^\infty \binom{b}{n}(-1)^nt^n \text{ and } 
	(1+t)^b=\sum_{n=0}^\infty \binom{b}{n}t^n,$$
we obtain
{\allowdisplaybreaks
\begin{align*}
g(t) & = 2^{b+1}t-6 \sum_{n=0}^\infty \binom{b}{2n+1}t^{2n+1}+2t\sum_{n=0}^\infty \binom{b}{2n}t^{2n}\\
h(t) & = 2\left(2^b+\sum_{n=0}^\infty t^{2n}\left( \binom{b}{2n}-3\binom{b}{2n+1}\right) \right) \\
h'(t) & = 4\sum_{n=1}^\infty nt^{2n-1} \frac{b(b-1)\ldots(b-(2n-1))}{(2n+1)!}(2n+1-3(b-2n)) \\
& = 4t\left( \frac{b(b-1)}{2}(3-b) \right. \\
& \quad + \left. \sum_{n=2}^\infty nt^{2n-2} \frac{b(b-1)\ldots(b-(2n-1))}{(2n+1)!}(2n+1-3(b-2n))\right) \\
& = 4tb(b-1)(3-b) \left(\frac{1}{2}+r(t)\right) 
\end{align*}
}
where
$$r(t) = \sum_{n=2}^\infty a_nt^{2n-2} ,\,
	a_n= 3n\frac{(b-2)(b-4)\ldots(b-(2n-1))}{(2n+1)!}(b-(8n+1)/3).$$
\begin{lemma}
It holds $r(t)>-1/2$ for all $b>1$ and $0<t<1$.
\end{lemma}
\begin{proof}
For all $n\geq 2$ the coefficient $a_n$ consists of an even number of factors 
containing the variable $b$. 
If $b\leq 2$, then all these factors are negative. Hence, $r(t)\geq0$ for $0<t<1$ and $b\leq 2$.

We assume $b>2$.
Let $N\subset\N$ be the set of indices of negative coefficients in the power series expansion of $r$, 
i.e., $N=\{ n: n\geq 2, a_n<0\}$. 
Since the inequality $(8n+1)/3 > 2n-1$ holds for all $n\in\N$, the factor $b-(8n+1)/3$ is the smallest
factor of the numerator of the coefficient $a_n$, i.e.,  $N\subset \{Ên > (3b-1)/8 \}$. 
We decompose $N$ into two disjoint subsets. 
Set $N_0 :=\{ n\in N: 2n-1<b\}$ and $N_1:=\{ n\in N : 2n-1\geq b\}$.
\begin{itemize}
\item If $n\in N_1$, then 
	\begin{align*}
	|a_n| & < n\frac{(2n-3)!}{(2n+1)!}(8n+1-3b) < \frac{n(8n-5)}{(2n-2)(2n-1)2n(2n+1)}\\
	& < \frac{1}{4}\cdot \frac{8n-4}{(n-1)(2n-1)(2n+1)}=  \frac{1}{(n-1)(2n+1)} <\frac{1}{2}\frac{1}{n(n-1)}.
	\end{align*}
\item If $n\in N_0$, then $|a_n|$ can be very large.
	But the set $N_0$ is finite for any fixed $b$. More precisely,
	$N_0 = \{ k_0,\ldots,k_1\}$ with $k_0=\lfloor (3b-1)/8 \rfloor +1$ and $k_1=\lceil (b+1)/2\rceil -1$.
	
	We want to show that $a_{k_0-1-l}+a_{k_0-1+l}\geq 0$ for all $1\leq l \leq k_1-k_0+1$
	and, consequently, $\sum_{n=2}^{k_1}a_nt^n \geq 0$ for $0<t<1$. 
	Unfortunately, this works only for $b>15$.
	We investigate the remaining cases separately.
	As for large $b$, some of the negative summands $a_nt^n$ with $n\in N_0$ 
	can be compensated by positive summands $a_nt^n$ with $n<k_0$.
	Other negative summands fulfill $a_n>-\frac{1}{2n(n-1)}$.
	We aim at an estimate of the form
	$$\sum_{n=2}^{k_1} a_nt^n \geq -\frac{1}{2}\sum_{n=2}^{k_1}\frac{1}{n(n-1)}.$$
	Here are the sets $N_0$ for the exceptional cases:
	{\allowdisplaybreaks
	\begin{align*}
	b\leq 3 & \Rightarrow N_0 = \emptyset \\
	3<b\leq 5 &\Rightarrow N_0 = \{Ê2\} \\
	5<b<17/3 & \Rightarrow N_0 = \{Ê2, 3 \} \\
	17/3\leq b \leq 7  & \Rightarrow N_0 = \{Ê3 \} \\
	7< b <25/3 &  \Rightarrow N_0 = \{Ê3, 4 \} \\
	25/3 \leq  b \leq 9 & \Rightarrow N_0 = \{Ê4 \} \\
	9 < b < 11 & \Rightarrow N_0 = \{Ê4, 5 \} \\
	b=11 & \Rightarrow N_0 = \{Ê5\} \\
	11< b \leq 13 & \Rightarrow N_0=\{Ê5, 6\} \\
	13 < b < 41/3 & \Rightarrow N_0=\{ 5,6,7\} \\
	41/3 \leq b \leq 15 & \Rightarrow N_0 =\{ 6,7\}
	\end{align*}
	}
	\begin{itemize}
	\item If $3< b \leq 5$, then 
		\begin{align*}
		|a_2| & = 2\frac{b-2}{5!}(17-3b) = \frac{-3b^2+23b-34}{3\cdot 4 \cdot 5}
		= -\frac{1}{20}\left(b^2-\frac{23}{3}b+\frac{34}{3}\right) \\
		& \leq \left. -\frac{1}{20}\left(b^2-\frac{23}{3}b+\frac{34}{3}\right)\right\vert_{b=\frac{23}{6}}
		 = \frac{121}{20\cdot 36} <\frac{1}{4} .
		\end{align*}
	\item If $5 < b < 17/3$, then
	\begin{align*}
		|a_2| & = 2\frac{b-2}{5!}(17-3b) < 2\frac{4}{5!}2 = \frac{2}{15}< \frac{1}{4}\\
		|a_3| & = 3\frac{(b-2)(b-4)(b-5)}{7!}(25-3b) < 3\frac{4\cdot 2\cdot 1}{7!} 10 
			= \frac{1}{14} < \frac{1}{12} .
		\end{align*}
	\item If $17/3\leq b \leq 19/3 \leq 7$, then
		$$|a_3| = 3\frac{(b-2)(b-4)(b-5)}{7!}(25-3b) < 
			3\frac{\frac{13}{3}\cdot \frac{7}{3} \cdot \frac{4}{3}}{7!} 8 < \frac{1}{12}.$$
		If $19/3< b \leq 7$, then
		$$\left\vert\frac{a_2}{a_3}\right\vert = \frac{2\cdot 7!(3b-17)}{3\cdot 5!(b-4)(b-5)(25-3b)}
			>  \frac{28\cdot 2}{3\cdot 2\cdot 6}>1.$$
	\item If $7< b \leq 25/3$, then
		\begin{align*}
		|a_4| & = 4\frac{(b-2)(b-4)(b-5)(b-6)(b-7)}{9!}(33-3b) \\
		& < 4\frac{\frac{19}{3}\cdot\frac{13}{3}\cdot \frac{10}{3}\cdot\frac{7}{3}\cdot \frac{4}{3}}{9!}12 
			= \frac{19\cdot 13}{3^8}< \frac{1}{24} \\
		\left\vert\frac{a_2}{a_3}\right\vert & = \frac{2\cdot 7!(3b-17)}{3\cdot 5!(b-4)(b-5)(25-3b)}
			>  \frac{28\cdot 4}{5\cdot 4\cdot 4}>1.
		\end{align*}
	\item If $9< b \leq 10$, then
		$$|a_5| = 5\frac{(b-2)(b-4)\ldots(b-9)}{11!}(41-3b) 
			\leq 5\frac{8\cdot 6 \cdots 1}{11!} 14 < \frac{1}{40}.$$
		If $10<b<11$, then
		\begin{align*}
		\left\vert\frac{a_3}{a_5}\right\vert &
			= \frac{3\cdot 11! (3b-25)}{5\cdot 7!(b-6)(b-7)(b-8)(b-9)(41-3b)} \\
		& > \frac{3\cdot 8\cdot 9 \cdot 10 \cdot 11\cdot 2}{5\cdot 5! \cdot 11} >1
		\end{align*}
	\item If $13 < b < 41/3$, then
		\begin{align*}
		|a_7| & = 7 \frac{(b-2)(b-4)\ldots(b-13)}{15!}(57-3b) \\
		& < \frac{7\cdot 35\cdot 29\cdot 26\cdot 23\cdot 20\cdot 17\cdot 14\cdot 11\cdot 8\cdot 5\cdot 2 
		\cdot 4}{3^{11} \cdot 2\cdot 3 \cdot 4 \cdot 5\cdot 6 \cdot 7 \cdot 8 \cdot 9 \cdot 10 \cdot 11 \cdot 
		12 \cdot 13 \cdot 14 \cdot 15} 18 \\
		& = \frac{ 7\cdot 29\cdot 23\cdot 17}{3^{15} } <\frac{1}{2\cdot 7 \cdot 6} 
		\end{align*}
	\item If $k_0\geq 4$ and $l\in \N$ such that $1\leq l\leq k_1-k_0+1$ and $k_0-1-l>1$, then
		\begin{align*}
		\left\vert\frac{a_{k_0-1-l}}{a_{k_0-1+l}}\right\vert & 
			= \frac{k_0-1-l}{k_0-1+l} \cdot \frac{(2k_0+2l-1)!}{(2k_0-2l-1)!} \cdot
			\frac{\prod_{j=4}^{2k_0-2l-3}(b-j)}{\prod_{j=4}^{2k_0+2l-3}(b-j)} \cdot
			\frac{|\alpha-8l|}{|\alpha+8l|}  \\
		&= 2(k_0-1-l)(2k_0+2l-1)\cdot \frac{|\alpha-8l|}{|\alpha+8l|}
			\prod_{j=2k_0-2l-2}^{2k_0+2l-3} \frac{j}{b-j} \\
		& \geq 2(k_0-1-l)(2k_0+2l-1) \prod_{j=2k_0-2l-2}^{2k_0+2l-3} \frac{j}{b-j} ,
		\end{align*}
		since  
		$k_0-1-l\geq k_0-1-(k_1-k_0+1) = 2k_0-k_1-2>1$
		and $0\geq \alpha:=8(k_0-1)+1-3b >8$.
		Moreover, the condition $3b<8k_0+1$ implies the inequality $b-j< 2k_0-j+(2k_0+1)/3$.
		Hence,
		$$\prod_{j=2k_0-2l-2}^{2k_0+2l-3} \frac{j}{b-j} > 
			\frac{(2k_0-2l-2)\cdots (2k_0+2l-3)}{(\frac{2k_0+1}{3}+3-2l)\ldots (\frac{2k_0+1}{3}+2+2l)}
			\geq 1,$$
		if $k_0\geq 4$, since the inequality $2k_0+2l-3 \geq 2+2l+(2k_0+1)/3$ holds for $k_0\geq 4$. 
		Hence, $|a_{k_0-1-l}|\geq |a_{k_0-1+l}|$.
	\end{itemize}
	If $b\geq 25/3$, then $k_0\geq 4$.
	If $b>15$, then the inequality $k_0-1-l>1$ holds for all $1\leq l\leq k_1-k_0+1$, since
	\begin{align*}
	k_0-1-l & \geq k_0-1-(k_1-k_0+1) = 2k_0-k_1-2\\
	& > \frac{3b-1}{4}-\frac{b+1}{2}-2 = \frac{1}{4}(b-11)>1.
	\end{align*}
	Note that $k_0-1-l>1$ holds for all $1\leq l\leq k_1-k_0+1$ also if
	$41/3\leq b \leq 15$, $11\leq b \leq 13$ or $25/3\leq b \leq 9$.
\end{itemize}
We obtain
$$r(t)>-\frac{1}{2}\sum_{n=2}^\infty\frac{1}{n(n-1)}t^{2n-2} 
	> -\frac{1}{2}\sum_{n=2}^\infty \frac{1}{n(n-1)} 
	= -\frac{1}{2}\sum_{n=2}^\infty  \frac{1}{n-1}-\frac{1}{n} = -\frac{1}{2},$$
since $1>t^{n_1}>t^{n_2}$ if $0<t<1$ and $n_1<n_2$.
\end{proof}
\subsection{Boundary components}
\label{subsec:Rand}
It follows from Corollary  \ref{folg:1Punkt-optimal} that the values of $f$ on the sets 
$M\cap \{ x=0\}$ and $M\cap\{ x=\sqrt{3}/4\}$ are strictly larger than the global minimum of $f$.
The minimum of $f$ on the boundary component  $M\cap\{ x=y\}$ is attained exactly at one point. 
This is $(x_1,y_1)$ with  $x_1=y_1:=\sqrt{3}/6$.
The remaining boundary component of $M$ is $M\cap\{ y=0\}$.
Set $R(x):=f(x,0)$. It holds
$$R(x)= 2x^p+\left(\frac{\sqrt{3}}{2}-x\right)^p,\, R'(x) = 2px^{p-1}-p\left(\frac{\sqrt{3}}{2}-x\right)^{p-1}$$
and
$$R''(x) = p(p-1)\left( x^{p-2}+\left(\frac{\sqrt{3}}{2}-x\right)^{p-2}\right)>0 \text{ for } 0<x<\sqrt{3}/4.$$
The equation $R'(x)=0$ has exactly one solution,
$$R'(x) = 0 \Leftrightarrow 2^{b}x=\frac{\sqrt{3}}{2}-x \Leftrightarrow x= \frac{\sqrt{3}}{2(2^b+1)}=:x_0
	\text{ with } b=\frac{1}{p-1}.$$
Since the function $R$ is strictly convex, the minimum of $R$ on $M\cap\{ y=0\}$ 
is attained only at $x_0$ and 
\begin{equation}
R(x_0) = \frac{\sqrt{3}^p}{2^{p-1}(1+2^b)^{p-1}}.
\end{equation}
\begin{lemma}
It holds
\begin{itemize}
\item $R(x_0)>f(x_1,y_1)$ if and only if $b<1$ or $b>3$, i.e.,  $p>2$ or $p<4/3$.
\item $R(x_0)<f(x_1,y_1)$ if and only if  $1<b<3$, i.e.,  $4/3<p<2$.
\item $R(x_0)=f(x_1,y_1)$ if and only if $b=1$ or $b=3$, i.e., $p=2$ or $p=4/3$.
\end{itemize}
\end{lemma}
\begin{proof}
$$\frac{\sqrt{3}^p}{2^{p-1}(1+2^b)^{p-1}} <\frac{1}{2^{p-1}} \Leftrightarrow \sqrt{3}^p <(1+2^b)^{p-1}
	\Leftrightarrow \sqrt{3}\sqrt{3}^b <1+2^b$$
Set $v(b) := 1+2^b-\sqrt{3}^{b+1}$. It is easy to check that $v(1) =  1+2-\sqrt{3}^2=0$ and 
$v(3) = 1+8-\sqrt{3}^4 =0$.
Furthermore, the function $v$ has at most one critical point, since 
$$v'(b) = 2^b\ln 2-\sqrt{3}^{b+1}\ln \sqrt{3}= 0 \Leftrightarrow
	\frac{\sqrt{3}\ln \sqrt{3}}{\ln 2} = \left(\frac{2}{\sqrt{3}}\right)^b.$$
Now, the assertions follow from $v(0) >0$. 
\end{proof}
\subsection{Global minimum of $f$ on $M$}
\label{subsec:globalesMinimum}
Comparing the minima on the boundary components $\{ x=y\}$ and $\{ y=0\}$ yields 
Theorem \ref{satz:OptimumParallellen} and Theorem \ref{satz:OptimumMittelsenkrechten}.
\begin{theorem}
\label{satz:OptimumParallellen}
If $1<p<4/3$ or $2<p$, then 
$$R(x_0) = f(x_0,0) =  \frac{\sqrt{3}^p}{2^{p-1}(1+2^b)^{p-1}} \text{ with }
	x_0 = \frac{\sqrt{3}}{2(2^{\frac{1}{p-1}}+1)}$$ 
is the global minimum of $f$ on $M$. It is attained only at $(x_0,0)\in M$.
\end{theorem}
If $1\leq p<4/3$ or $p>2$, then the global  minimum of $f$ on $M$ is reached by a line parallel to
the triangle side $p_1p_2$ with distance $x_0$ to that side.
This optimal line has normal vector $n=(0,1)$ and is given by the equation $\langle n, q\rangle = x_0$. 
Note that
$$\lim_{p\to 1} x_0=0,\quad \lim_{p\to 4/3} x_0 =\frac{\sqrt{3}}{18},\quad 
	\lim_{p\to 2} x_0 = \frac{\sqrt{3}}{6} = x_1.$$ 
\begin{theorem}
\label{satz:OptimumMittelsenkrechten}
If $4/3<p<2$, then $f(\sqrt{3}/6,\sqrt{3}/6) = 2^{1-p}$ is the global minimum of $f$ on $M$. 
It is attained only at $(\sqrt{3}/6,\sqrt{3}/6)\in M$.
\end{theorem}
If  $4/3<p<2$, then the optimal line in $M$ is independent of $p$.
This optimal line contains $p_1$ and is perpendicular to the line through $p_2$ and $p_3$.
\begin{theorem}
If $p=2$ or $p=4/3$, then $2^{1-p}$ is the global minimum of $f$ on $M$.
The minimum is attained at an one-dimensional submanifold of $M$:
\begin{itemize}
\item Let $p=2$ and $(x,y)\in M$.  It holds $f(x,y)=2^{-1}$ if and only if $x=\sqrt{3}/6$.
\item Let $p=4/3$ and $(x,y)\in M$. It holds $f(x,y)=2^{-1/3}$ if and only if 
	$$x= \frac{1+36y^2}{6\sqrt{3}(1+4y^2)} \text{ and }Ê0\leq y \leq \frac{\sqrt{3}}{6}.$$ 
\end{itemize}
\end{theorem}
\begin{proof}
If $b=1$ or $b=3$, then equation (\ref{eq:kritischer Punkt, Bedingung in t})  becomes trivial.
Hence, $(x,y)\in M$ is a critical point of $f$ if and only if equation (\ref{eq:kritischer Punkt, x=x(y)})
and $0<y<\sqrt{3}/6$ are satisfied (see Corollary \ref{folg:kritischer Punkt, Intervall fuer y} and 
Corollary \ref{folg:kritischer Punkt, zwei Gleichungen}).
\begin{itemize}
\item If $b=1$, i.e., $p=2$, then equation (\ref{eq:kritischer Punkt, x=x(y)}) becomes
	$$ x= 
	y \frac{\left(1+2\sqrt{3}y\right)+\left(1-2\sqrt{3}y\right)}{\left(1+2\sqrt{3}y\right)-\left(1-2\sqrt{3}y\right)} 
	= y \frac{2}{4\sqrt{3}y} = \frac{1}{2\sqrt{3}} = \frac{\sqrt{3}}{6} = x_1.$$
	Note that $(0,x_1)$ is the center of mass of the triangle $D$ (see Figure \ref{Abb:OptiRedukt}).
	The function defined by $y\mapsto f(x_1,y)$ is constant. 
	The values $f(x_1,\sqrt{3}/6)$ and $f(x_1,0)$ are the global minima on the boundary components
	$\{ x=y\}\cap M$ respectively $\{ y=0\}\cap M$, since
	$\lim_{p\to 2} x_0 = \sqrt{3}/6$ (see subsection \ref{subsec:Rand}).
\item If $b=3$, i.e.,  $p=4/3$, then equation (\ref{eq:kritischer Punkt, x=x(y)}) becomes
	\begin{align*} 
	x= &
	y \frac{\left(1+2\sqrt{3}y\right)^3+\left(1-2\sqrt{3}y\right)^3}{\left(1+2\sqrt{3}y\right)^3
	-\left(1-2\sqrt{3}y\right)^3} = y \frac{2+72y^2}{12\sqrt{3}y+48\sqrt{3}y^3} 
	= \frac{1+36y^2}{6\sqrt{3}(1+4y^2)}\\
	x(y) & = \frac{\sqrt{3}}{18}\left( 9-\frac{8}{1+4y^2}\right)
	\end{align*}
	The function $x(y)$ is strictly increasing for  $0\leq y$.
	The function defined by $y\mapsto f(x(y),y)$ is constant.
	Since $x(\sqrt{3}/6) = \sqrt{3}/6$ and $x(0)=\sqrt{3}/18=\lim_{p\to 4/3} x_0$, the values
	$f(x(\sqrt{3}/6),\sqrt{3}/6)$ and $f(x(0),0)$ coincide with the global minima 
	on the boundary components $\{ x=y\}\cap M$ respectively $\{ y=0\}\cap M$ 
	(see subsection \ref{subsec:Rand}).
\end{itemize}
\end{proof}
\section{Summary}
Applying the symmetry group of the triangle $D$ to the minima of $f:M\to \R$ found in subsection
\ref{subsec:globalesMinimum}  we obtain all optimal lines:
\subsection{$1\leq p<4/3$ and $2<p\leq \infty$}
If $1\leq p<4/3$ or $2<p\leq \infty$, then there exist exactly three $L^p$-optimal lines. 
These are the lines intersecting the triangle $D$, parallel to one of the sides of the triangle $D$ 
with distance $x_0(p)=\sqrt{3}(2^{\frac{p}{p-1}}+1)^{-1}$ to that side
(Figure \ref{Abb:lOptBg3} and Figure \ref{Abb:OptBk1}).

Any of this three lines is invariant under the reflection in the perpendicular bisector of the triangle side 
parallel to that line. The set of optimal lines is generated by the rotations 
around the center of mass of $D$ through $2\pi/3$ and the optimal line found in subsection \ref{subsec:globalesMinimum}
for $1<p<4/3$ and f"ur $2<p\leq\infty$.
\begin{figure}
\noindent
\begin{minipage}[b]{.49\linewidth}
  \mbox{\scalebox{.28}{\includegraphics{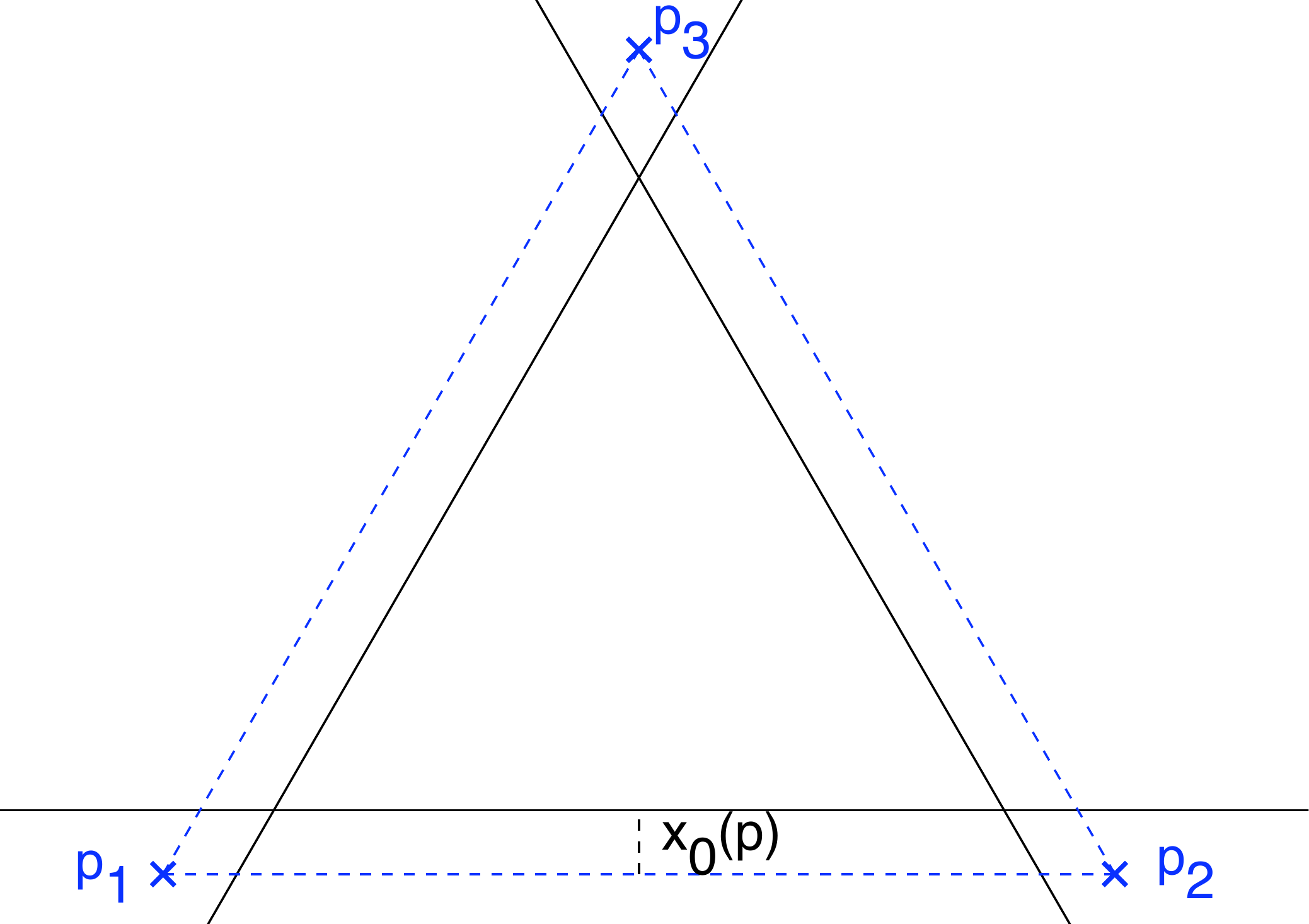}}}
\caption{$1\leq p<4/3$}
\label{Abb:lOptBg3}
\end{minipage}
\begin{minipage}[b]{.49\linewidth}
  \mbox{\scalebox{.28}{\includegraphics{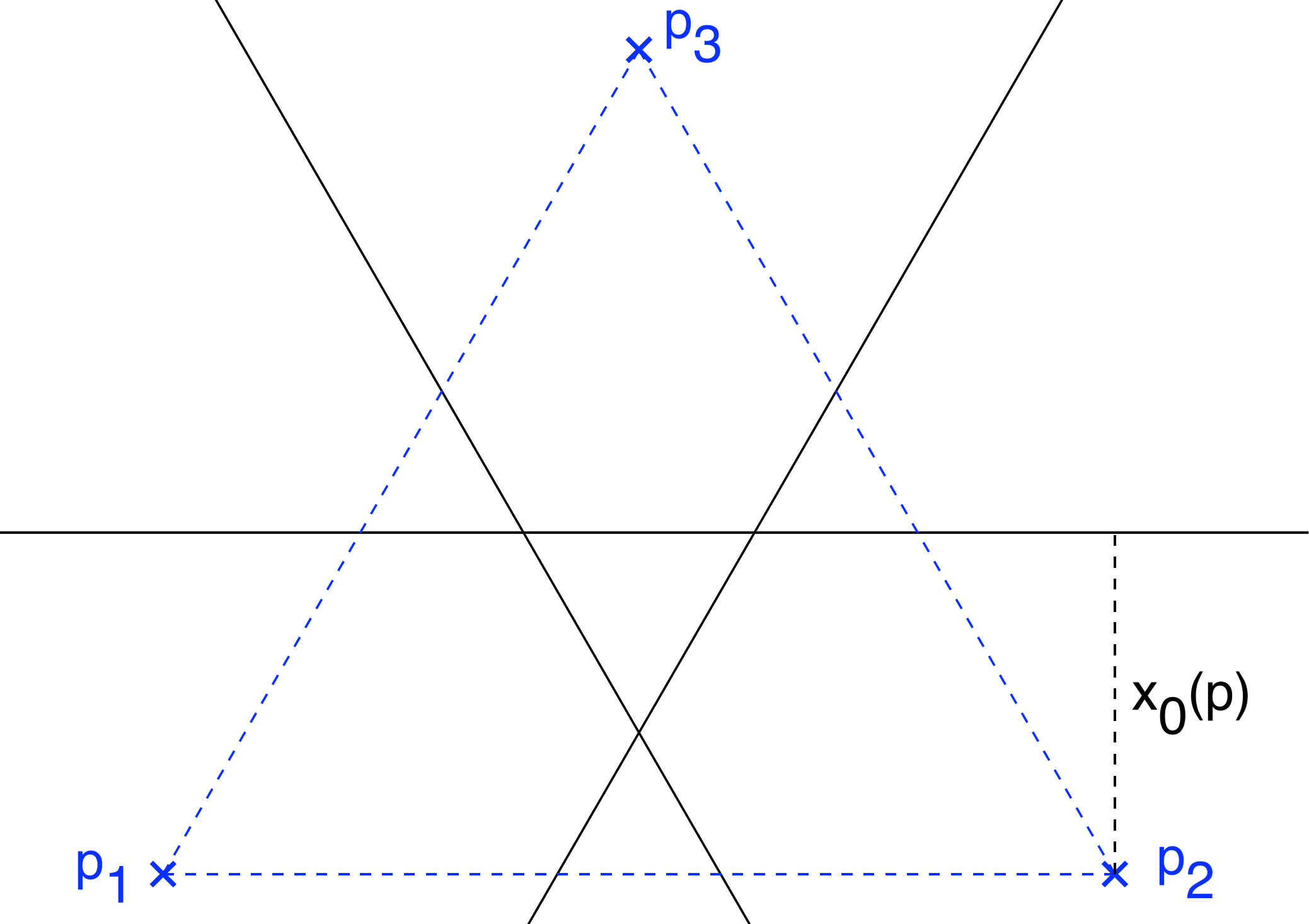}}}
\caption{$2<p\leq\infty$}
\label{Abb:OptBk1}
\end{minipage}
\end{figure}
\subsection{$4/3<p<2$}
If $4/3<p<2$, then there are exactly three $L^p$-optimal lines.
These are the lines containing one vertex of the triangle and parallel to the triangle side 
opposite to that vertex (see Figure \ref{Abb:OptBk3g1}). 

These optimal lines are invariant under the reflections in the symmetry group of $D$.
Again, the set of optimal lines is generated by the rotations in the symmetry group of $D$ and
the optimal line found in subsection \ref{subsec:globalesMinimum}  for $4/3<p<2$.
\subsection{$p=2$}
A line $g$ is $L^2$-optimal if and only if $g$ contains the center of mass of the triangle $D$
(see Figure \ref{Abb:OptB1}).
The set of all $L^2$-optimal lines arises as the orbit of the optimal lines 
found in subsection \ref{subsec:globalesMinimum}  for $p=2$ by the symmetry group of $D$. 
If $y\ne 0,\sqrt{3}/6$, then the orbit of $(\sqrt{3}/6,y)$ consists of six lines.
If $y=0$ or $y=\sqrt{3}/6$, then the orbit consists of three lines.

The green lines in Figure \ref{Abb:OptB1}, that are obtained with $y=\sqrt{3}/6$,
coincide with the optimal lines for  $4/3<p<2$. 
The red lines in Figure \ref{Abb:OptB1} are obtained with $y=0$ and as limits $p\to 2$ of
optimal lines for  $p>2$ in Figure \ref{Abb:OptBk1}.
\subsection{$p=4/3$}
The set of $L^{4/3}$-optimal lines is most conveniently described as the orbit of 
the set of optimal lines found in subsection \ref{subsec:globalesMinimum} for $p=4/3$
by the symmetry group of the triangle $D$.
If $y\ne 0,\sqrt{3}/6$, then the orbit of $(x(y),y)$ consists of six lines 
(see Figures \ref{Abb:OptB3Y2Zent}, \ref{Abb:OptB3Y5Zent} and \ref{Abb:OptB3Y7Zent}). 
If $y=0$ (Figure \ref{Abb:OptB3Y0}) or $y=\sqrt{3}/6$ (Figure \ref{Abb:OptBk3g1}),
then the orbit consists of three lines.
The $L^{4/3}$-optimal lines in Figure \ref{Abb:OptB3Y0} are also the limits $p\to 4/3$ 
of $L^p$-optimal lines with $p<4/3$ (see Figure \ref{Abb:lOptBg3}).
Of course, the $L^{4/3}$-optimal lines in Figure \ref{Abb:OptBk3g1} are the limits
$p\to 4/3$ of $L^p$-optimal lines for $p>4/3$.
\begin{figure}
\noindent
\begin{minipage}[b]{.49\linewidth}
  \mbox{\scalebox{.28}{\includegraphics{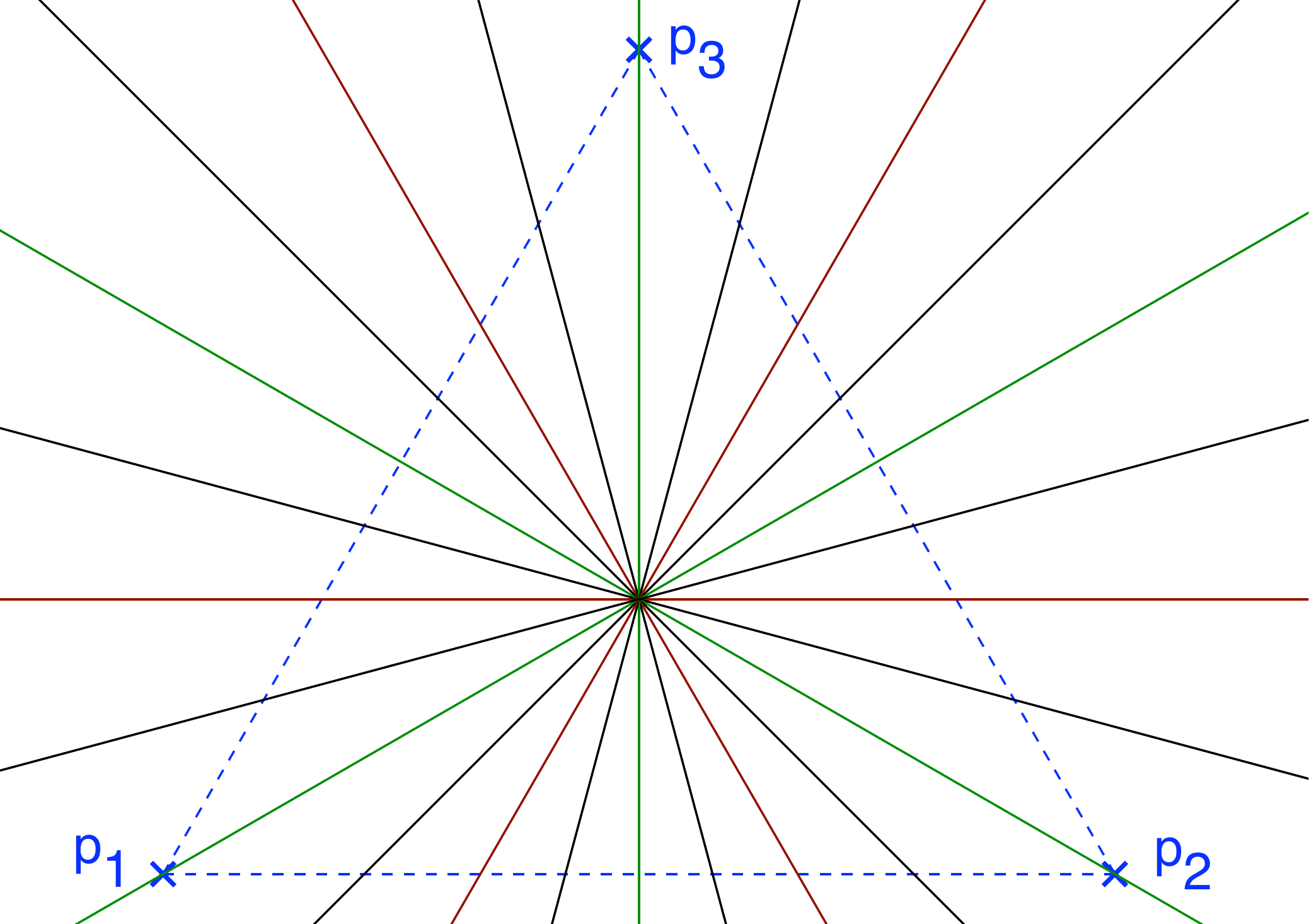}}}
\caption{$p=2$}
\label{Abb:OptB1}
\end{minipage}
\begin{minipage}[b]{.49\linewidth}
  \mbox{\scalebox{.28}{\includegraphics{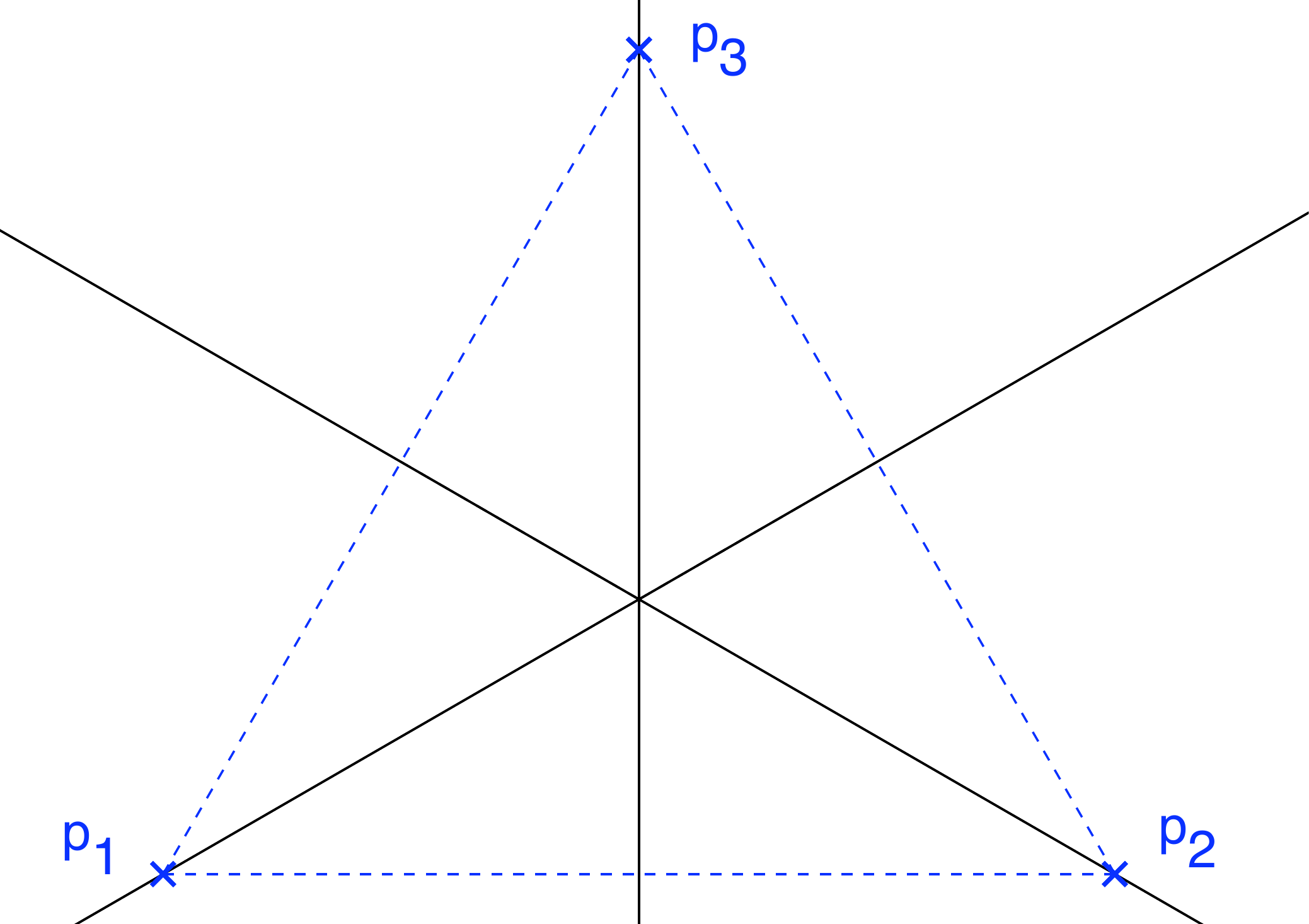}}}
\caption{$4/3<p<2$}
\label{Abb:OptBk3g1}
\end{minipage}
\end{figure}
\begin{figure}
\noindent
\begin{minipage}[b]{.49\linewidth}
  \mbox{\scalebox{.28}{\includegraphics{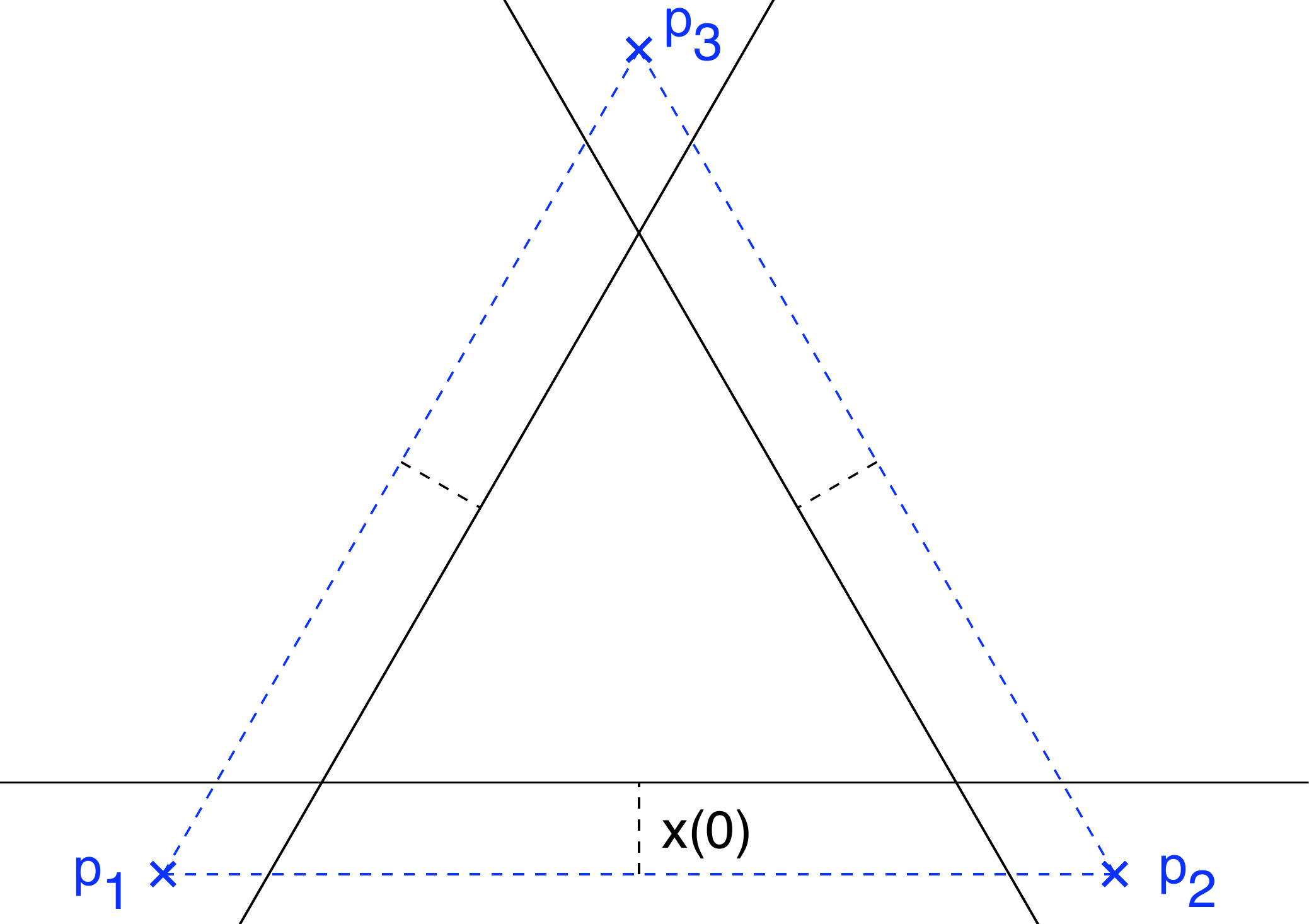}}}
\caption{$p=4/3$, $y=0$}
\label{Abb:OptB3Y0}
\end{minipage}
\begin{minipage}[b]{.49\linewidth}
  \mbox{\scalebox{.28}{\includegraphics{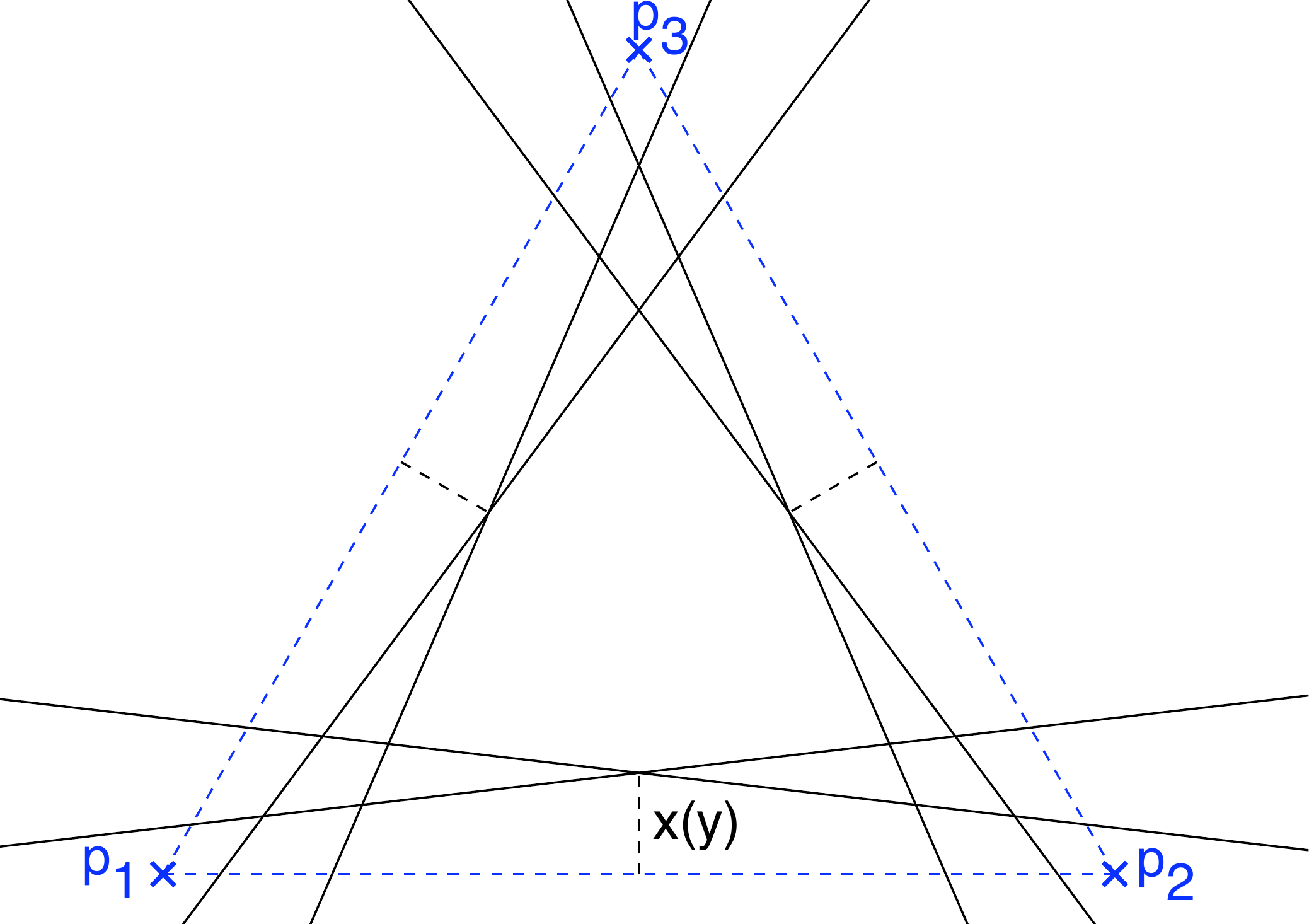}}}
\caption{$p=4/3$, $y=2\sqrt{3}/60$}
\label{Abb:OptB3Y2Zent}
\end{minipage}
\end{figure}
\begin{figure}
\noindent
\begin{minipage}[b]{.49\linewidth}
  \mbox{\scalebox{.28}{\includegraphics{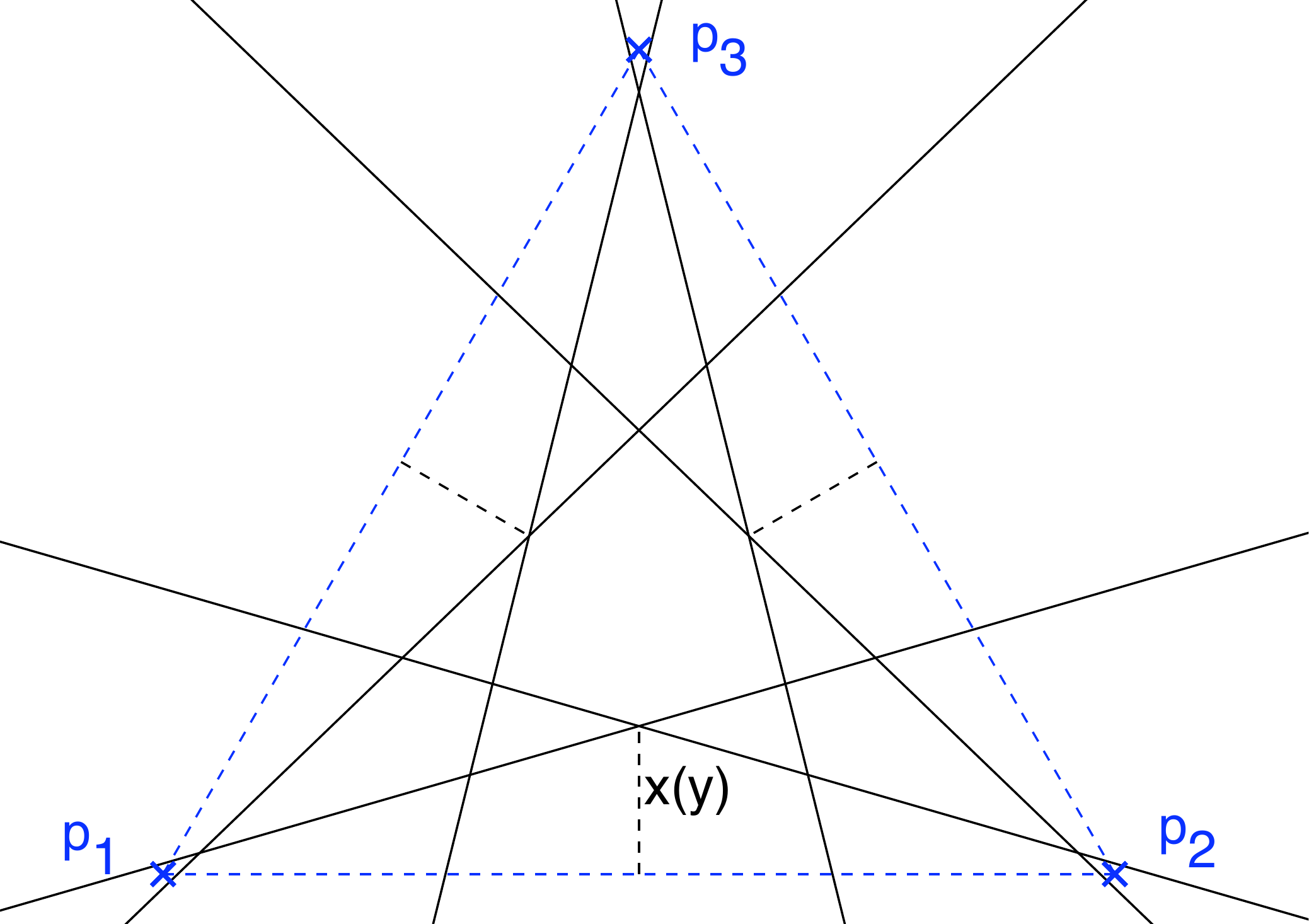}}}
\caption{$p=4/3$, $y=5\sqrt{3}/60$}
\label{Abb:OptB3Y5Zent}
\end{minipage}
\begin{minipage}[b]{.49\linewidth}
  \mbox{\scalebox{.28}{\includegraphics{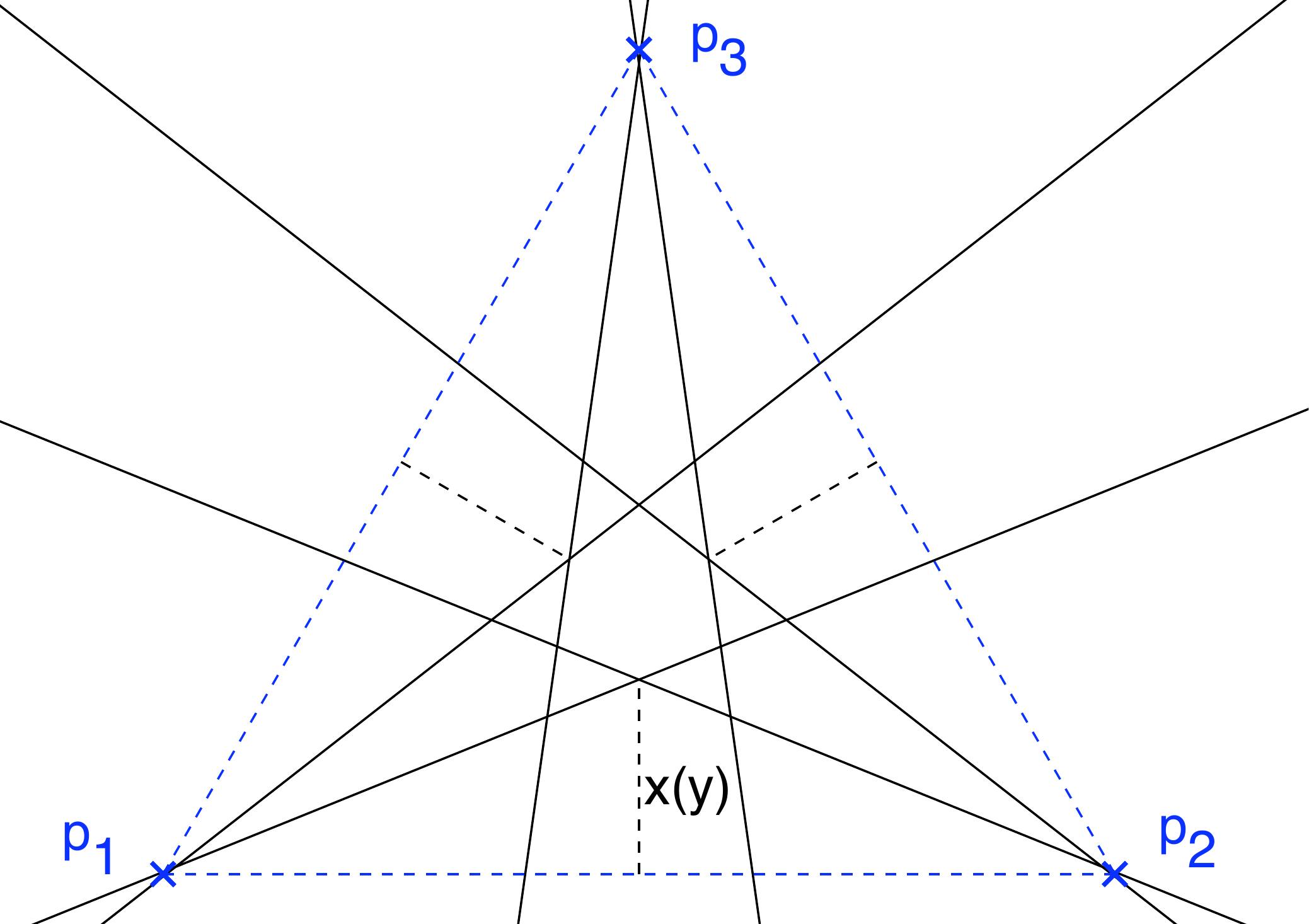}}}
\caption{$p=4/3$, $y=7\sqrt{3}/60$}
\label{Abb:OptB3Y7Zent}
\end{minipage}
\end{figure}

\end{document}